\documentclass{amsart}

\usepackage{amssymb}
\usepackage{amstext}
\usepackage{amsmath}
\usepackage{amscd}
\usepackage{amsthm}
\usepackage{amsfonts}
\usepackage{enumerate}
\usepackage[dvipdfmx]{graphicx}
\usepackage{color}
\usepackage{here} 
\usepackage{bm} 
\usepackage{mathrsfs}
\usepackage{mathtools}
\usepackage{latexsym}
\usepackage{booktabs}
\usepackage{fancyhdr}
\usepackage{subcaption}
\usepackage{tikz}
\usepackage{tikz-cd}
\usetikzlibrary{arrows}
\usetikzlibrary{decorations.markings}
\usetikzlibrary{patterns,decorations.pathreplacing}
\usepackage{mathtools}
\usepackage{pdfpages}
\usepackage[section]{algorithm}
\usepackage{algpseudocodex}

\usepackage{hyperref}
\hypersetup{
  colorlinks=true,
  citecolor=red,
  linkcolor=blue
}

\makeatletter
  
\@addtoreset{equation}{section}
\makeatother

\newcommand{\calA}{\mathcal{A}}

\newcommand{\calF}{\mathcal{F}}

\newcommand{\calR}{\mathcal{R}}

\newcommand{\ZZ}{\mathbb{Z}}

\newcommand{\kk}{\Bbbk}

\newcommand{\ab}{\mathbf{a}}

\newcommand{\bb}{\mathbf{b}}

\newcommand{\cb}{\mathbf{c}}

\newcommand{\wb}{\mathbf{w}}

\newcommand{\xb}{\mathbf{x}}

\newcommand{\Hom}{\operatorname{Hom}}

\newcommand{\Gro}{Gr\"{o}bner }
\newcommand{\Pluecker}{Pl\"{u}cker }

\newcommand{\conv}{\mathrm{conv}}

\newcommand{\cone}{\mathrm{cone}}

\newcommand{\GL}{\mathrm{GL}}

\newcommand{\Stab}{\mathrm{Stab}}

\newcommand{\intnonneg}{\mathbb{Z}_{\geq 0}}

\def\opn#1#2{\def#1{\operatorname{#2}}} 
\opn\Cl{Cl} \opn\conv{conv} \opn\deg{deg} \opn\rank{rank} \opn\Spec{Spec} \opn\Stab{Stab} \opn\aff{aff} \opn\div{div} \opn\GL{GL}
\opn\cone{cone} \opn\End{End} \opn\Hom{Hom} \opn\mod{mod} \opn\gldim{gldim} \opn\pdim{pdim} \opn\diag{diag} \opn\vert{vert}
\opn\Block{Block} \opn\Pyr{Pyr} \opn\max{max} \opn\min{min} \opn\ini{in} \opn\rev{rev} \opn\ker{ker} \opn\lat{lat} \opn\pull{pull} \opn\rev{rev}
\opn\Gale{Gale} \opn\sign{sign} \opn\supp{supp}
\opn\cok{coker} \opn\core{core} \opn\star{star}
\opn\pd{pd} \opn\Soc{Soc} \opn\Ap{Ap} \opn\Sym{Sym}
\opn\PF{PF} \opn\t{t} \opn\F{F} \opn\e{e}
\opn\m{m} \opn\G{G} \opn\g{g} \opn\H{H}
\opn \embdim{embdim} \opn\var{var}
\opn{\mult}{mult} \opn{\emb}{emb}
\opn{\gap}{Gap}
\opn{\gr}{gr}

\newtheorem{thm}{Theorem}[section]
\newtheorem{theorem}[thm]{Theorem}

\newtheorem{lemma}[thm]{Lemma}

\newtheorem{proposition}[thm]{Proposition}

\theoremstyle{definition}

\newtheorem{definition}[thm]{Definition}

\newtheorem{example}[thm]{Example}

\theoremstyle{remark}
\newtheorem{remark}[thm]{Remark}

\title{Khovanskii bases of subalgebras arising from finite distributive lattices}

\author{Akihiro Higashitani}
\address{Department of Pure and Applied Mathematics, Graduate School of Information Science and Technology, Osaka University, Osaka, Japan}
\email{higashitani@ist.osaka-u.ac.jp}

\author{Koji Matsushita}
\address{Department of Pure and Applied Mathematics, Graduate School of Information Science and Technology, Osaka University, Osaka, Japan}
\email{k-matsushita@ist.osaka-u.ac.jp}

\author[K. Tani]{Koichiro Tani}
\address{Department of Pure and Applied Mathematics, Graduate School of Information
Science and Technology, Osaka University, Osaka, Japan}
\email{tani-k@ist.osaka-u.ac.jp}
\date{}

\subjclass[2020]{
Primary: 13F65;
Secondary; 05E40, 13P10, 14M25, 14M15. 
}
\keywords{Khovanskii basis, SAGBI basis, Distributive lattice, Hibi ideal, Generalized snake poset, $(2+2)$-free poset}

\begin{document}
\begin{abstract}
  The notion of Khovanskii bases was introduced by Kaveh and Manon~\cite{KhovanskiiBasis}.
  It is a generalization of the notion of SAGBI bases for a subalgebra of polynomials.
  The notion of SAGBI bases was introduced by Robbiano and Sweedler~\cite{RobbianoSubalgebra}
  as an analogue of \Gro bases in the context of subalgebras.
  A Hibi ideal is an ideal of a polynomial ring that arises from a distributive lattice.
  For the development of an analogy of the theory of Hibi ideals and \Gro bases within the framework of subalgebras, 
  in this paper, we investigate when the set of the polynomials associated with a distributive lattice forms
  a Khovanskii basis of the subalgebras it generates.
  We characterize such distributive lattices and their underlying posets. 
  In particular, generalized snake posets and $\{(2+2),(1+1+1)\}$-free posets appear as the characterization. 
\end{abstract}
\maketitle

\section{Introduction}
Let $\kk$ be a field and let $S$ be a finitely generated $\kk$-domain.
We call an abelian group $(\Delta, +, \leq)$ equipped with a total order $\leq$
a \textit{linearly ordered group} if the total order $\leq$ satisfies that $\delta_1 \leq \delta_2$
implies $\delta_1 + \delta_3 \leq \delta_2 + \delta_3$ for all
$\delta_1, \delta_2, \delta_3 \in \Delta$.
We consider a discrete valuation $v : S \to \Delta$.
The valuation $v$ gives a $\Delta$-filtration $T_{v} = (T_{v \geq \delta})_{\delta \in \Delta}$
on $S$, where $T_{v \geq \delta}$ (resp. $T_{v > \delta}$) is defined by
\begin{equation*}
  T_{v \geq \delta} \coloneqq \{f \in S \mid v(f) \geq \delta\} \cup \{0\} \quad
  (\text{resp. } T_{v > \delta} \coloneqq \{f \in S \mid v(f) > \delta\} \cup \{0\}).
\end{equation*}
The associated graded ring $\gr_{v}(S)$ induced by $v$ and $T_v$ is
\begin{equation*}
  \gr_{v}(S) = \bigoplus_{\delta \in \Delta} T_{v \geq \delta} / T_{v > \delta}.
\end{equation*}
For $f \in S \setminus \{0\}$, we can consider its image $\bar{f}$ in $\gr_{v}(S)$,
namely the image of $f$ in $T_{v \geq \delta} / T_{v > \delta}$ where $v(f) = \delta$.
For a $\kk$-subalgebra $R$, a set $\calF \subset R$ is a \textit{Khovanskii basis}
with respect to $v$ if the image of $\calF$ in the associated graded ring $\gr_{v}(S)$
forms a set of generators of the image of $R$ in $\gr_{v}(S)$.
For more details, see \cite{KhovanskiiBasis}. 

Let $S = \kk[x_1, \ldots, x_n]$ be the polynomial ring in $n$ variables over $\kk$.
We use an abbreviation of monomials $x_1^{a_1} \cdots x_n^{a_n}$ with $x^\ab$ for $\ab = (a_1, \ldots, a_n) \in (\intnonneg)^n$. 
A \textit{monomial order} $\preceq$ is a total order on the set of
all monomials in $S$ such that $x^{\ab} \preceq x^{\bb}$ implies $x^{\ab+\cb} \preceq x^{\bb+\cb}$ 
for each $x^{\ab}, x^{\bb}, x^{\cb} \in S$ and $1$ is the least monomial.
The largest term of $f \in S$ with respect to $\preceq$ is called the \textit{initial term} of $f$ with respect to $\preceq$
and we denote it as $\ini_{\preceq}(f)$. Let $R$ be a finitely generated $\kk$-subalgebra of $S$.
A $\kk$-vector space spanned by $\{\ini_{\preceq}(f) \mid f \in R\}$
is called the \textit{initial algebra} of $R$ and denoted by $\ini_{\preceq}(R)$. 
A subset $\calF$ of $R$ is called a \textit{SAGBI basis} of $R$ if
$\ini_{\preceq}(\calF) \coloneqq \{\ini_{\preceq}(f) \mid f \in \calF\}$
generates the $\kk$-subalgebra $\ini_{\preceq}(R)$.
The word ``SAGBI" is introduced by Robbiano and Sweedler~\cite{RobbianoSubalgebra} and  stands
for "Subalgebra Analogue to Gr\"{o}bner Bases for Ideal", and
Kaveh and Manon~\cite{KhovanskiiBasis} generalized the notion of SAGBI bases by far with Khovanskii bases.

Let $P$ be a finite poset with a partial order $\leq_P$. A \textit{poset ideal} of $P$ is a subset $\alpha$ of $P$
satisfying that, whenever $a \in \alpha$ and $b \in P$ with $b \leq_P a$, one has $b \in \alpha$.
Let $L(P)$ be the set of poset ideals of $P$. Since the empty set and $P$ itself
are elements of $L(P)$, and
since $\alpha \cap \beta$ and $\alpha \cup \beta$ belong to $L(P)$ for all
$\alpha,\beta \in L(P)$, the set $L(P)$ forms a finite lattice ordered by inclusion, 
where the join and meet operations correspond to union and intersection, respectively.
Moreover, there is a one-to-one correspondence between finite distributive lattices $L(P)$ and finite posets $P$. It is known as Birkhoff's representation theorem \cite{Birkhoff1967}.

In this paper, we study Khovanskii bases of subalgebras arising from finite distributive lattices $L(P)$.
Let $S = \kk[x_{\alpha} \mid \alpha \in L(P)]$,
$\calF_{L(P)} := \{f_{\alpha, \beta} = x_{\alpha}x_{\beta}-x_{\alpha \wedge \beta}x_{\alpha \vee \beta} \in S \mid \alpha, \beta \in L(P) \text{ are incomparable}\}$, and $\calR(L(P))=\kk[\calF_{L(P)}] \subset S$. 
This subalgebra $\calR(L(P))$ is a subalgebra analogue of the Hibi ideal of a distributive lattice.
For more details on Hibi ideals, see~\cite[Chapter 6]{BinomialIdeals}.
A monomial order $\preceq$ on $S$ is said to be \textit{compatible} 
if $\ini_{\preceq}(f_{\alpha, \beta}) = x_{\alpha}x_{\beta}$ for all $\alpha, \beta \in L(P)$
for which $\alpha$ and $\beta$ are incomparable.
Such a monomial order exists because any partial order has a linear extension. 
In fact, let $\leq_{L(P)}^{*}$ be a linear extension of $\leq_{L(P)}$ and consider the degree reverse lexicographic order $\preceq$ whose variables satisfy $x_{\alpha} \preceq x_{\beta}$
if $\alpha \leq_{L(P)}^{*} \beta$ holds in $L(P)$.
It then follows that this monomial order $\preceq$ is compatible.

Our goal is to characterize the class of distributive lattices such that $\calF_{L(P)}$ forms a Khovanskii basis of $\calR(L(P))$ 
with respect to a compatible monomial order.
To this end, it suffices to consider only the case where $P$ is irreducible with respect to ordinal sum, i.e., $P$ cannot be written as the ordinal sum of two proper subposets of $P$ (see Lemma~\ref{lemma:main lemma}). 
Here, the \textit{ordinal sum} of two posets $(P,\leq_P)$ and $(Q,\leq_Q)$ is the poset $P\oplus Q$ on $P\cup Q$ such that $a \leq_{P\oplus Q} b$ if (i) $a,b\in P$ and $a\leq_P b$, (ii) $a,b\in Q$ and $a\leq_Q b$, or (iii) $a\in P$ and $b\in Q$.
The following is our main result:
\begin{theorem}\label{theorem:main theorem}
Let $P$ be a finite poset and assume that $P$ is irreducible with respect to ordinal sum.
Then the following conditions are equivalent:
\begin{enumerate}
    \item[{\em (1)}] $\calF_{L(P)}$ forms a Khovanskii basis of $\calR(L(P))$ with respect to a compatible monomial order; 
    \item[{\em (2)}] $P$ is $\{(2+2),(1+1+1)\}$-free;
    \item[{\em (3)}] $L(P)$ is a generalized snake poset.
\end{enumerate}
\end{theorem}
\begin{remark}
As a corollary of Theorem~\ref{theorem:main theorem}, we find that a
generalized snake poset corresponds to a $\{(2+2),(1+1+1)\}$-free poset
that is irreducible with respect to the ordinal sum under the correspondence
provided by Birkhoff's representation theorem. 
\end{remark}
We call a poset is \textit{$(2+2)$-free} (resp. \textit{$(1+1+1)$-free}) if the poset does not contain
a poset described in Figure~\ref{fig:forbidden subposet 1} (resp. Figure~\ref{fig:forbidden subposet 2}) as a subposet.
We say that a poset is \textit{$\{(2+2),(1+1+1)\}$-free} if it is both $(2+2)$-free and $(1+1+1)$-free. 

\begin{figure}[ht]
  \begin{minipage}{0.48\columnwidth}
    \centering
    {\scalebox{0.45}{
      \begin{tikzpicture}[line width=0.05cm]
          \coordinate (N1) at (0,0); 
          \coordinate (N2) at (0,2);
          \coordinate (N3) at (2,0);
          \coordinate (N4) at (2,2);
          \draw  (N1)--(N2); 
          \draw  (N3)--(N4);
          \draw [line width=0.05cm, fill=white] (N1) circle [radius=0.15];
          \draw [line width=0.05cm, fill=white] (N2) circle [radius=0.15];
          \draw [line width=0.05cm, fill=white] (N3) circle [radius=0.15];
          \draw [line width=0.05cm, fill=white] (N4) circle [radius=0.15];
        \end{tikzpicture}
      }}
      \caption{A poset ``$(2+2)$''}
      \label{fig:forbidden subposet 1}
    \end{minipage}
    \begin{minipage}{0.5\columnwidth}
      \centering
      {\scalebox{0.45}{
        \begin{tikzpicture}[line width=0.05cm]
        \coordinate (N1) at (-2,0);
        \coordinate (N2) at (0,0);
        \coordinate (N3) at (2,0);
        \draw [line width=0.05cm, fill=white] (N1) circle [radius=0.15];
        \draw [line width=0.05cm, fill=white] (N2) circle [radius=0.15];
        \draw [line width=0.05cm, fill=white] (N3) circle [radius=0.15];
        \node at (0,-1) {$\;$};
        \node at (0,1) {$\;$};
      \end{tikzpicture}
        }}
        \caption{A poset ``$(1+1+1)$''}
        \label{fig:forbidden subposet 2}
    \end{minipage}
\end{figure}

$(2+2)$-free posets have been studied, e.g., in~\cite{BogartFreePoset,DukesVitKubitzkeCompositionmatrix, FishburnIntervalOrder}. 
For the investigation of $\{(2+2),(1+1+1)\}$-free posets that are irreducible with respect to ordinal sum, 
we use a one-to-one correspondence between $(2+2)$-free posets and composition matrices
constructed in~\cite{DukesVitKubitzkeCompositionmatrix}
(see Section~\ref{section:generalized snake posets and (2+2) and (1+1+1)-free posets}).
Generalized snake posets (see Definition~\ref{def:generalized snake poset}) have been also studied, e.g., in~\cite{lee2024generalizedsnakeposetsorder, generalizedsnakeposet1}. 

\medskip

This paper is organized as follows. 
In Section~\ref{section:Preliminaries}, we provide the basics of Khovanskii(and SAGBI) basis, toric ideals of edge rings, and \Pluecker algebras.
In Section~\ref{section:Applied Khovanskii basis criterion and Examples},
we apply the tools introduced in Section~\ref{section:Preliminaries} to our main targets and provide the necessary examples. 
In Section~\ref{section:generalized snake posets and (2+2) and (1+1+1)-free posets}, we show a one-to-one correspondence between generalized snake posets
and $\{(2+2),(1+1+1)\}$-free posets that are irreducible with respect to ordinal sum. 
Finally, we prove Theorem~\ref{theorem:main theorem} in the last part of Section~\ref{section:generalized snake posets and (2+2) and (1+1+1)-free posets}.

\subsection*{Acknowledgement}
The first author is partially supported by JSPS KAKENHI Grant Number JP24K00521 and JP21KK0043. 
The second author is partially supported by Grant-in-Aid for JSPS Fellows Grant JP22J20033.

\section{Preliminaries}\label{section:Preliminaries}
In this section, we introduce definitions and results about Khovanskii (and SAGBI) basis,
toric ideals of edge rings, and \Pluecker algebras.

\begin{definition}\label{definition:valuation}
  Let $(\Delta, +, \leq)$ be a linearly ordered group. A function $v : S \setminus \{0\} \to \Delta$
  is a \textit{valuation} over $\kk$ if it satisfies the following axioms:
  \begin{enumerate}
    \item If $f, g \in S \setminus \{0\}$ and $f+g \neq 0$, then we have $v(f+g) \geq \min\{v(f), v(g)\}$.
    \item If $f, g \in S \setminus \{0\}$, then we have $v(fg) = v(f) + v(g)$.
    \item If $f \in S \setminus \{0\}$ and $c \in \kk \setminus \{0\}$, then we have $v(cf) = v(f)$.
  \end{enumerate}
\end{definition}

Throughout this paper, we assume $S$ is a polynomial ring and we use valuations induced by
monomial orders. More precisely, let $S = \kk[x_1, \ldots, x_n]$ and let $\preceq$ be a
monomial order on $S$. Let $\Delta = \ZZ^n$ and let the equipped order $\leq$ be
a total order on $\ZZ^n$ by setting
\begin{equation*}
  \ab_1 \leq \ab_2 \stackrel{\mathrm{def}}{\iff} \xb^{\ab_1} \succeq \xb^{\ab_2}
\end{equation*}
for $\ab_1, \ab_2 \in \ZZ^n$. This gives a valuation $v : S \setminus \{0\} \to \ZZ^n$
over $\kk$ by putting $v(f) = \ab$ for $f \in S \setminus \{0\}$ with
$\ini_{\preceq}(f) = c x^{\ab}$. We use this valuation Throughout this paper for given
polynomial ring $S$ and monomial order $\preceq$. Then, the image of $f \in S$ in
\begin{equation*}
  \gr_{v}(S) = \bigoplus_{\ab \in \ZZ^n} T_{v \geq \ab} / T_{v > \ab}
\end{equation*}
is the initial term $\ini_{\preceq}(f)$. For a finitely generated $\kk$-subalgebra $R \subset S$,
the image of $R$ in $\gr_{v}(S)$ is $\ini_{\preceq}(R)$. Therefore, the set $\calF \subset R$
is a Khovanskii basis with respect to $v$ induced by $\preceq$ if and only if all initial
terms of $\calF$ generate $\ini_{\preceq}(R)$. Of course, it means $\calF$ is a SAGBI
basis with respect to $\preceq$. In what follows, we say $\calF$ is a Khovanskii basis of
$R$ with respect to a monomial order $\preceq$ if $\calF$ is a Khovanskii basis with respect to $v$ induced by $\preceq$.

Algorithm~\ref{algorithm:subduction} is a modification of~\cite[Algorithm 11.1]{SturmfelsLectureNote}.
In~\cite{SturmfelsLectureNote}, Khovanskii bases are are referred to as canonical bases.

\begin{algorithm}[ht]
  \caption{(The subduction algorithm)}\label{algorithm:subduction}
  \begin{algorithmic}
    \Require $\calF = \{f_1, \ldots, f_s\} \subset S = \kk[x_1, \ldots, x_n]$, $f \in S$
    \Ensure $q \in \kk[\calF], r \in S$ such that $f = q + r$
    \State $q \coloneqq 0; r \coloneqq 0$
    \State $p \coloneqq f$
    \While{$p \notin \kk$}
      \State find $i_1, i_2, \ldots, i_s \in \intnonneg$ and $c \in \kk\backslash\{0\}$ such that
      \begin{align}\label{eqn:initial representation}
	 \tag{$\ast$}
       \ini_{\preceq}(p) = c \cdot \prod_{j=1}^s{\ini_{\preceq}(f_j)}^{i_j}.
     \end{align}
      \If {representation~(\ref{eqn:initial representation}) exists}
      \State $q \coloneqq q + c \cdot f_1^{i_1} \cdot f_2^{i_2} \cdots f_s^{i_s}$
      \State $p \coloneqq p - c \cdot f_1^{i_1} \cdot f_2^{i_2} \cdots f_s^{i_s}$
      \Else
      \State $r \coloneqq r + \ini_{\preceq}{p}$
      \State $p \coloneqq p - \ini_{\preceq}{p}$
      \EndIf
    \EndWhile
    \State \Return $q, r$
  \end{algorithmic}
\end{algorithm}

Let $S=\kk[x_1,\ldots,x_n]$ and fix a monomial order $\preceq$ on $S$. 
Let $f_1,\ldots,f_s \in S$, let $\ab_1, \ab_2, \ldots, \ab_s \in (\intnonneg)^n$ with $\ini_{\preceq}(f_i) = x^{\ab_i}$ for each $i$, 
and let $\calA = (\ab_1, \ab_2, \ldots, \ab_s)$ be the $n \times s$-matrix whose columns are $\ab_i$'s. 
The toric ideal $I_{\calA}$ of $\calA$ is the kernel of a $\kk$-algebra homomorphism 
\begin{equation*}
  \kk[X_1, \ldots, X_s] \rightarrow S, \quad X_i \mapsto x^{\ab_i}. 
\end{equation*}
We can determine whether $\calF$ forms a Khovanskii basis as follows. 
\begin{proposition}[See~{\cite[Corollary 11.5]{SturmfelsLectureNote}}]\label{prop:Khovanskii criterion}
Let $\{p_1, p_2, \ldots, p_t\}$ be generators of the toric ideal $I_{\calA}$. 
Then $\calF$ forms a Khovanskii basis of $R$ if and only if Algorithm~\ref{algorithm:subduction} reduces
$p_i(f_1, f_2, \ldots, f_s)$ to an element of $\kk$ by $\calF$ for each $i$. 
\end{proposition}

To apply Proposition~\ref{prop:Khovanskii criterion} to $\calF_{L(P)}$ for a distributive lattice
$L(P)$, we introduce the notion of edge rings of graphs and describe their toric ideals.
For more details on toric ideals of edge rings, see~\cite[Section 5.3]{BinomialIdeals}.

Let $G$ be a finite simple graph on a vertex set $[n] = \{1, \ldots, n\}$ with an edge set $E(G) = \{e_1, \ldots, e_s\}$. 
The \textit{edge ring} of $G$ is the $\kk$-subalgebra of $S$ generated by $\{x_ix_j \mid \{i, j\} \in E(G)\}$, denoted by $\kk[G]$. 
The toric ideal of $\kk[G]$, denoted by $I_G$, is the kernel of the $\kk$-algebra homomorphism
\begin{equation*}
  \kk[X_1, \ldots, X_s] \to \kk[G], \quad X_{l} \mapsto \xb^{e_l}, 
\end{equation*}
where $\xb^{e_l} = x_ix_j$ for $e_l = \{i,j\}$.
Given an even closed walk $\Gamma = (e_{i_1}, e_{i_2}, \ldots, e_{i_{2q}})$ of $G$ with each $e_k \in E(G)$,
we write $p_{\Gamma}$ for the binomial
\begin{equation*}
  p_{\Gamma} = \prod_{k=1}^{q}X_{i_{2k-1}} - \prod_{k=1}^{q}X_{i_{2k}}
\end{equation*}
belonging to $I_G$. We abbreviate $\prod_{k=1}^{q}X_{i_{2k-1}}$ to $p_{\Gamma}^{(+)}$
and $\prod_{k=1}^{q}X_{i_{2k}}$ to $p_{\Gamma}^{(-)}$. An even closed walk $\Gamma$ of $G$
is called \textit{primitive} if there exists no even closed walk $\Gamma'$ of $G$ with
$p_{\Gamma'} \neq p_{\Gamma}$ for which $p_{\Gamma'}^{(+)}$ divides $p_{\Gamma}^{(+)}$, and $p_{\Gamma'}^{(-)}$ divides $p_{\Gamma}^{(-)}$.

\begin{lemma}[{\cite[Lemma 3.1]{OhsugiHibi1999}}]\label{lemma:generators of toric ideal}
  A binomial $f \in I_G$ is primitive if and only if there exists a primitive even closed walk
  $\Gamma$ of $G$ such that $f = p_{\Gamma}$. In particular, the toric ideal $I_G$ is generated
  by those binomials $p_{\Gamma}$, where $\Gamma$ is a primitive even closed walk of $G$.
\end{lemma}


The characterization of primitive even closed walks is given in \cite[Lemma 3.2]{OhsugiHibi1999}.
In particular, when the graph $G$ is bipartite, its primitive even closed walks are characterized as follows:

\begin{lemma}[See {\cite[Corollary 5.12]{BinomialIdeals}}]\label{corollary:toric ideal of bipartite graph}
    Let $G$ be a bipartite graph. Then every primitive even closed walk is an even cycle.
    In particular, the toric ideal $I_G$ is generated by those binomials $f_C$, where $C$ is an even cycle of $G$.
\end{lemma}

Finally, we introduce \Pluecker algebras and their Khovanskii bases. 
It turns out that some of $\calR(L(P))$ can be regarded as a \Pluecker algebra. 
For more details on \Pluecker algebras, see, e.g., \cite[Chapter 14]{MillerSturmfels} and \cite[Chapter 11]{SturmfelsLectureNote}. 
\begin{definition}
  Let $1 \leq k < d$ and let $X = (x_{ij})_{1 \leq i \leq k, 1 \leq j \leq d}$
  be a $k \times d$-matrix of indeterminates. Given
  $I \in \mathbf{I}_{k,d} \coloneqq \{I \subset [d] \mid |I|=k\}$,
  let $X_I$ denote the $k \times k$-submatrix of $X$ whose columns are indexed by $I$. 
  We define the subalgebra of the polynomial ring $\kk[x_{ij} \mid 1 \leq i \leq k, 1 \leq j \leq d]$ as follows: 
  \begin{equation*}
    A_{k,d} \coloneqq \kk[\det(X_I) \mid I \in \mathbf{I}_{k,d}].
  \end{equation*}
  The $\kk$-algebra $A_{k,d}$ is called the \textit{\Pluecker algebra}.
\end{definition}

\begin{proposition}[See, e.g.,~{\cite[Theorem 14.11]{MillerSturmfels}}]\label{proposition:plucker}
The generating set $\{\det(X_I) \mid I \in \mathbf{I}_{k,d}\}$ forms a Khovanskii basis of $A_{k,d}$ with respect to a monomial order 
such that the monomial of the initial term of each minor $\det(X_I)$ is the product of diagonal entries of $X_I$. 
\end{proposition}

\section{Application of Khovanskii basis criterion}\label{section:Applied Khovanskii basis criterion and Examples}
We apply Proposition~\ref{prop:Khovanskii criterion} to $\calF_{L(P)}$
for a distributive lattice $L(P)$. If a monomial order is compatible,
then the toric ideal is equal to the defining ideal of
\begin{equation*}
  \kk[\ini_{\preceq}(\calF_{L(P)})] = \kk[x_{\alpha}x_{\beta} \mid \alpha, \beta \in L(P) \text{ are incomparable}].
\end{equation*}
Let $L(P) = \{\alpha_1, \ldots, \alpha_n\}$. 
We define the graph $G_{L(P)}$ as a graph on a vertex set $[n]$ 
with an edge set $\{\{i, j\} \mid \alpha_i, \alpha_j \text{ are incomparable in } L(P)\}$.
The graph $G_{L(P)}$ is called the \textit{co-comparability graph} of $L(P)$.
We see that the subalgebra $\kk[\ini_{\preceq}(\calF_{L(P)})]$ under consideration coincides with the edge ring $\kk[G_{L(P)}]$ of $G_{L(P)}$. 
Let $\{\Gamma_1, \ldots, \Gamma_t\}$ be the set of all primitive even closed walks of $G_{L(P)}$ and 
consider the associated binomials $p_{\Gamma_1}, \ldots, p_{\Gamma_t} \in \kk[X_{\{i,j\}} \mid \alpha_i, \alpha_j \text{ are incomparable in } L(P)]$. 
Then $I_{G_{L(P)}}=(p_{\Gamma_1}, \ldots, p_{\Gamma_t})$. 

We have a Khovanskii basis criterion for $\calF_{L(P)}$ as follows.
\begin{lemma}\label{lemma:Khovanskii criterion for calF_L(P)}
  The set of polynomials $\calF_{L(P)}$ forms a Khovanskii basis of $\calR(L(P))$ with respect to a compatible monomial order 
  if and only if Algorithm~\ref{algorithm:subduction} reduces the polynomial $p_{\Gamma_i}(f_{\alpha_j, \alpha_k})_{\{j,k\} \in E(G_{L(P)})}$ 
  obtained by assigning $f_{\alpha_j, \alpha_k}$ to each variable $X_{\{j,k\}}$ of $p_{\Gamma_i}$ 
  to an element of $\kk$ by $\calF_{L(P)}$ for each $i=1, \ldots, t$.
\end{lemma}
\begin{proof}
  This lemma follows from Proposition~\ref{prop:Khovanskii criterion} and Lemma~\ref{lemma:generators of toric ideal}. 
\end{proof}

The following lemma will play an essential role in the proof of our main theorem. 
\begin{lemma}\label{lemma:main lemma}
  Fix a compatible monomial order. 
  Let $\Gamma_1,\ldots,\Gamma_s$ be primitive even closed walks of $G_{L(P)}$ and 
  $L_i$ the minimal sublattice of $L(P)$ containing $\{\alpha_i \mid i \in V(\Gamma_i)\}$ for each $i$. 
  Then $\calF_{L(P)}$ forms a Khovanskii basis of $\calR(L(P))$ if and only if $\calF_{L_i}$ forms a Khovanskii basis of $\calR(L_i)$ for each $i$.
\end{lemma}
\begin{proof}
  We prove that $\calF_{L(P)}$ does not form a Khovanskii basis of $\calR(L(P))$ if and only if
  $\calF_{L_i}$ does not form a Khovanskii basis of $\calR(L_i)$ for some $i$. 

  Assume $\calF_{L(P)}$ does not form a Khovanskii basis. Then there exists $p_{\Gamma_i}$ such that the polynomial
  \begin{equation}\label{equation:unsubduceable polynomial}
    p_{\Gamma_i}(f_{\alpha_j, \alpha_k})_{\{j,k\} \in E(G_{L(P)})}
  \end{equation}
  cannot be reduced to an element of $\kk$. 
  Since $G_{L_i}$ is an induced subgraph of $G_{L(P)}$ and $\Gamma_i$ is an even closed walk of $G_{L_i}$,
  $p_{\Gamma_i}$ is a binomial in $I_{G_{L_i}} \subset \kk[X_{\{j,k\}} \mid \{j,k\} \in E(G_{L_i})]$. 
  Thus, the polynomial~(\ref{equation:unsubduceable polynomial})
  cannot be reduced to an element of $\kk$ by $\calF_{L_i}$, either. 
  Hence, $\calF_{L_i}$ does not form a Khovanskii basis by Lemma~\ref{lemma:Khovanskii criterion for calF_L(P)}.

  Assume that $\calF_{L_i}$ does not form a Khovanskii basis of $\calR(L_i)$ for some $i$. 
  Then there exists a monomial $m \in \kk[x_{\alpha} \mid \alpha \in L_i]$ such that 
  $m \in \ini_{\preceq}(\calR(L_i)) \setminus \kk[\ini_{\preceq}(\calF_{L_i})]$. 
  Since $L_i$ is a sublattice of $L(P)$, the join and meet of any two elements in $L_i$ coincide with their join and meet in $L(P)$, respectively.
  Thus, we have $\calF_{L_i} \subset \calF_{L(P)}$ and $m \in \ini_{\preceq}(\calR(L_i)) \subset \ini_{\preceq}(\calR(L(P)))$. 
  However, $m \not\in \kk[\ini_{\preceq}(\calF_{L(P)})]$ because $m \not\in \kk[\ini_{\preceq}(\calF_{L_i})]$ 
  and monomials in $\ini_{\preceq}(\calF_{L(P)}) \setminus \ini_{\preceq}(\calF_{L_i})$ 
  have a variable corresponding to an element of $L(P) \setminus L_i$. 
  Hence, $\calF_{L(P)}$ does not form a Khovanskii basis.
\end{proof}
By this lemma, the distributive lattices under consideration can be minimized. 
In particular, we can assume that $P$ is irreducible with respect to ordinal sum. 
In fact, if $P$ is an ordinal sum of two posets $P_1$ and $P_2$, 
then the co-comparability graph $G_{L(P)}$ can be separated into two disconnected graphs $G_{L(P_1)}$ and $G_{L(P_2)}$,

In the rest of this section, we provide examples of distributive lattices $L(P)$ such that 
$\calF_{L(P)}$ forms, or does not form a Khovanskii basis. 
Those examples will be a part of the proof of Theorem~\ref{theorem:main theorem} 
thanks to Lemma~\ref{lemma:main lemma}. 

\begin{example}\label{example:dvisor lattice of 36}
  Let $L(P)$ be the divisor lattice of $36=2^2 \cdot 3^2$. Then $P$ is nothing but the poset in Figure~\ref{fig:forbidden subposet 1}. 
  We label the elements of $L(P)$ as shown in Figure~\ref{fig:divisor lattice of 36}, 
  and the co-comparability graph $G_{L(P)}$ is depicted in Figure~\ref{fig:co-comparability graph of divisor lattice}. 
 \begin{figure}[ht]
    \begin{minipage}{0.45\columnwidth}
    \centering
    {\scalebox{0.45}{
    \begin{tikzpicture}[line width=0.05cm]
      \coordinate (a1) at (0, 0);
      \coordinate (a2) at (-2, 2);
      \coordinate (a3) at (2, 2);
      \coordinate (a4) at (-4, 4);
      \coordinate (a5) at (0, 4);
      \coordinate (a6) at (4, 4);
      \coordinate (a7) at (-2, 6);
      \coordinate (a8) at (2, 6);
      \coordinate (a9) at (0, 8);
      
      \draw (a1) -- (a2);
      \draw (a1) -- (a3);
      \draw (a2) -- (a4);
      \draw (a2) -- (a5);
      \draw (a3) -- (a5);
      \draw (a3) -- (a6);
      \draw (a4) -- (a7);
      \draw (a5) -- (a7);
      \draw (a5) -- (a8);
      \draw (a6) -- (a8);
      \draw (a7) -- (a9);
      \draw (a8) -- (a9);

      \draw [line width=0.05cm, fill=white] (a1) circle [radius=0.15] node [below=2pt] {\Huge $\alpha_1$};
      \draw [line width=0.05cm, fill=white] (a2) circle [radius=0.15] node [below left=1pt] {\Huge $\alpha_2$};
      \draw [line width=0.05cm, fill=white] (a3) circle [radius=0.15] node [below right=1pt] {\Huge $\alpha_3$};
      \draw [line width=0.05cm, fill=white] (a4) circle [radius=0.15] node [left=2pt] {\Huge $\alpha_4$};
      \draw [line width=0.05cm, fill=white] (a5) circle [radius=0.15] node [below=8pt] {\Huge $\alpha_5$};
      \draw [line width=0.05cm, fill=white] (a6) circle [radius=0.15] node [right=2pt] {\Huge $\alpha_6$};
      \draw [line width=0.05cm, fill=white] (a7) circle [radius=0.15] node [above left=1pt] {\Huge $\alpha_7$};
      \draw [line width=0.05cm, fill=white] (a8) circle [radius=0.15] node [above right=2pt] {\Huge $\alpha_8$};
      \draw [line width=0.05cm, fill=white] (a9) circle [radius=0.15] node [above=2pt] {\Huge $\alpha_9$};
    \end{tikzpicture}
    }}
    \caption{$L(P)$}
    \label{fig:divisor lattice of 36}
    \end{minipage} 
    \begin{minipage}{0.45\columnwidth}
    \centering
    {\scalebox{0.45}{
    \begin{tikzpicture}[line width=0.05cm]
      \coordinate (a1) at (0, 0);
      \coordinate (a2) at (-2, 2);
      \coordinate (a3) at (2, 2);
      \coordinate (a4) at (-4, 4);
      \coordinate (a5) at (0, 4);
      \coordinate (a6) at (4, 4);
      \coordinate (a7) at (-2, 6);
      \coordinate (a8) at (2, 6);
      \coordinate (a9) at (0, 8);
      
      \draw (a2) -- (a3);
      \draw (a2) -- (a6);
      \draw (a3) -- (a4);
      \draw (a4) -- (a5);
      \draw (a4) to [out=15,in=165] (a6);
      \draw (a4) -- (a8);
      \draw (a5) -- (a6);
      \draw (a6) -- (a7);
      \draw (a7) -- (a8);

      \draw [line width=0.05cm, fill=white] (a1) circle [radius=0.15] node [below=2pt] {\Huge $1$};
      \draw [line width=0.05cm, fill=white] (a2) circle [radius=0.15] node [below left=1pt] {\Huge $2$};
      \draw [line width=0.05cm, fill=white] (a3) circle [radius=0.15] node [below right=1pt] {\Huge $3$};
      \draw [line width=0.05cm, fill=white] (a4) circle [radius=0.15] node [left=2pt] {\Huge $4$};
      \draw [line width=0.05cm, fill=white] (a5) circle [radius=0.15] node [below=4pt] {\Huge $5$};
      \draw [line width=0.05cm, fill=white] (a6) circle [radius=0.15] node [right=2pt] {\Huge $6$};
      \draw [line width=0.05cm, fill=white] (a7) circle [radius=0.15] node [above left=1pt] {\Huge $7$};
      \draw [line width=0.05cm, fill=white] (a8) circle [radius=0.15] node [above right=2pt] {\Huge $8$};
      \draw [line width=0.05cm, fill=white] (a9) circle [radius=0.15] node [above=2pt] {\Huge $9$};
    \end{tikzpicture}
    }}
    \caption{$G_{L(P)}$}
    \label{fig:co-comparability graph of divisor lattice}
    \end{minipage}
    \end{figure}

    For the co-comparability graph $G_{L(P)}$, we can find a primitive even closed walk
  \begin{equation*}
    \Gamma = \{\{2,3\}, \{3,4\}, \{4,6\}, \{2,6\}\}.
  \end{equation*}
  Therefore, $p_{\Gamma} = X_{\{2,3\}}X_{\{4,6\}} - X_{\{3,4\}}X_{\{2,6\}}$
  is one of generators of $I_{G_{L(P)}}$. Then
  \begin{align*}
    p_{\Gamma}(f_{\alpha_j, \alpha_k})_{\{j,k\} \in E(G_{L(P)})}
    &= f_{\alpha_2,\alpha_3}f_{\alpha_4,\alpha_6} - f_{\alpha_3,\alpha_4}f_{\alpha_2,\alpha_6} \\
    &= (x_{\alpha_2}x_{\alpha_3} - x_{\alpha_1}x_{\alpha_5})(x_{\alpha_4}x_{\alpha_6} - x_{\alpha_1}x_{\alpha_9}) \\
    &\quad - (x_{\alpha_3}x_{\alpha_4} - x_{\alpha_1}x_{\alpha_7})(x_{\alpha_2}x_{\alpha_6} - x_{\alpha_1}x_{\alpha_8}) \\
    &= -x_{\alpha_1}x_{\alpha_2}x_{\alpha_3}x_{\alpha_9} - x_{\alpha_1}x_{\alpha_4}x_{\alpha_5}x_{\alpha_6}
    + x_{\alpha_1}^2x_{\alpha_5}x_{\alpha_9} \\
    & \quad + x_{\alpha_1}x_{\alpha_3}x_{\alpha_4}x_{\alpha_8} + x_{\alpha_1}x_{\alpha_2}x_{\alpha_6}x_{\alpha_7}
    - x_{\alpha_1}^2x_{\alpha_7}x_{\alpha_8}
  \end{align*}
  No term can be written as a product of monomials in $\ini_{\preceq}(\calF_{L(P)})$
  because every term contains either $x_{\alpha_1}$ or $x_{\alpha_9}$, 
  while $\alpha_1$ and $\alpha_9$ are comparable to every other element of $L(P)$.
  Thus, this polynomial cannot be reduced to an element of $\kk$ by $\calF_{L(P)}$. 
  This implies that $\calF_{L(P)}$ does not form a Khovanskii basis by Lemma~\ref{lemma:main lemma}.
\end{example}
\begin{example}\label{example:boolean lattice of rank 3}
  Let $L(P)$ be the boolean lattice of rank $3$. Then $P$ is nothing but the poset in Figure~\ref{fig:forbidden subposet 2}. 
  We label the elements of $L(P)$ as shown in Figure~\ref{fig:boolean lattice of rank 3}, 
  and the co-comparability graph $G_{L(P)}$ is depicted in Figure~\ref{fig:co-comparability graph of boolean lattice}. 
  \begin{figure}[ht]
    \begin{minipage}{0.45\columnwidth}
    \centering
    {\scalebox{0.45}{
      \begin{tikzpicture}[line width=0.05cm]
        \coordinate (a1) at (0, 0);
        \coordinate (a2) at (-3, 2);
        \coordinate (a3) at (0, 2);
        \coordinate (a4) at (3, 2);
        \coordinate (a5) at (-3, 4);
        \coordinate (a6) at (0, 4);
        \coordinate (a7) at (3, 4);
        \coordinate (a8) at (0, 6);
        
        \draw (a1) -- (a2);
        \draw (a1) -- (a3);
        \draw (a1) -- (a4);
        \draw (a2) -- (a5);
        \draw (a2) -- (a6);
        \draw (a3) -- (a5);
        \draw (a3) -- (a7);
        \draw (a4) -- (a6);
        \draw (a4) -- (a7);
        \draw (a5) -- (a8);
        \draw (a6) -- (a8);
        \draw (a7) -- (a8);
  
        \draw [line width=0.05cm, fill=white] (a1) circle [radius=0.15] node [below=2pt] {\Huge $\alpha_1$};
        \draw [line width=0.05cm, fill=white] (a2) circle [radius=0.15] node [below left=1pt] {\Huge $\alpha_2$};
        \draw [line width=0.05cm, fill=white] (a3) circle [radius=0.15] node [below left=1pt] {\Huge $\alpha_3$};
        \draw [line width=0.05cm, fill=white] (a4) circle [radius=0.15] node [below right=1pt] {\Huge $\alpha_4$};
        \draw [line width=0.05cm, fill=white] (a5) circle [radius=0.15] node [above left=1pt] {\Huge $\alpha_5$};
        \draw [line width=0.05cm, fill=white] (a6) circle [radius=0.15] node [above left=1pt] {\Huge $\alpha_6$};
        \draw [line width=0.05cm, fill=white] (a7) circle [radius=0.15] node [above right=1pt] {\Huge $\alpha_7$};
        \draw [line width=0.05cm, fill=white] (a8) circle [radius=0.15] node [above=2pt] {\Huge $\alpha_8$};
      \end{tikzpicture}
    }}
    \caption{$L(P)$}
    \label{fig:boolean lattice of rank 3}
    \end{minipage} 
    \begin{minipage}{0.45\columnwidth}
    \centering
    {\scalebox{0.45}{
      \begin{tikzpicture}[line width=0.05cm]
        \coordinate (a1) at (0, 0);
        \coordinate (a2) at (-3, 2);
        \coordinate (a3) at (0, 2);
        \coordinate (a4) at (3, 2);
        \coordinate (a5) at (-3, 4);
        \coordinate (a6) at (0, 4);
        \coordinate (a7) at (3, 4);
        \coordinate (a8) at (0, 6);
        
        \draw (a2) -- (a3);
        \draw (a2) to [out=-30,in=-150] (a4);
        \draw (a2) -- (a7);
        \draw (a3) -- (a4);
        \draw (a3) -- (a6);
        \draw (a4) -- (a5);
        \draw (a5) -- (a6);
        \draw (a5) to [out=30,in=150] (a7);
        \draw (a6) -- (a7);
  
        \draw [line width=0.05cm, fill=white] (a1) circle [radius=0.15] node [below=2pt] {\Huge $1$};
        \draw [line width=0.05cm, fill=white] (a2) circle [radius=0.15] node [below left=1pt] {\Huge $2$};
        \draw [line width=0.05cm, fill=white] (a3) circle [radius=0.15] node [below left=1pt] {\Huge $3$};
        \draw [line width=0.05cm, fill=white] (a4) circle [radius=0.15] node [below right=1pt] {\Huge $4$};
        \draw [line width=0.05cm, fill=white] (a5) circle [radius=0.15] node [above left=1pt] {\Huge $5$};
        \draw [line width=0.05cm, fill=white] (a6) circle [radius=0.15] node [above left=1pt] {\Huge $6$};
        \draw [line width=0.05cm, fill=white] (a7) circle [radius=0.15] node [above right=1pt] {\Huge $7$};
        \draw [line width=0.05cm, fill=white] (a8) circle [radius=0.15] node [above=2pt] {\Huge $8$};
      \end{tikzpicture}
    }}
    \caption{$G_{L(P)}$}
    \label{fig:co-comparability graph of boolean lattice}
    \end{minipage}
    \end{figure}

  For the co-comparability graph $G_{L(P)}$, we can find a primitive even closed walk
  \begin{equation*}
    \Gamma = \{\{2,4\}, \{4,5\}, \{5,7\}, \{2,7\}\}.
  \end{equation*}
  Therefore, $p_{\Gamma} = X_{\{2,4\}}X_{\{5,7\}} - X_{\{4,5\}}X_{\{2,7\}}$ is one of generators of $I_{G_{L(P)}}$. Then
  \begin{align*}
    p_{\Gamma}(f_{\alpha_j, \alpha_k})_{\{j,k\} \in E(G_{L(P)})}
    &= f_{\alpha_2,\alpha_4}f_{\alpha_5,\alpha_7} - f_{\alpha_4,\alpha_5}f_{\alpha_2,\alpha_7} \\
    &= (x_{\alpha_2}x_{\alpha_4} - x_{\alpha_1}x_{\alpha_6})(x_{\alpha_5}x_{\alpha_7} - x_{\alpha_3}x_{\alpha_8})\\
    & \quad - (x_{\alpha_4}x_{\alpha_5} - x_{\alpha_1}x_{\alpha_8})(x_{\alpha_2}x_{\alpha_7} - x_{\alpha_1}x_{\alpha_8}) \\
    &= -x_{\alpha_2}x_{\alpha_3}x_{\alpha_4}x_{\alpha_8} - x_{\alpha_1}x_{\alpha_5}x_{\alpha_6}x_{\alpha_7}
    + x_{\alpha_1}x_{\alpha_3}x_{\alpha_6}x_{\alpha_8} \\
    & \quad + x_{\alpha_1}x_{\alpha_4}x_{\alpha_5}x_{\alpha_8} + x_{\alpha_1}x_{\alpha_2}x_{\alpha_7}x_{\alpha_8}
    - x_{\alpha_1}^2x_{\alpha_8}^2
  \end{align*}
  No term can be written as a product of monomials in $\ini_{\preceq}(\calF_{L(P)})$
  because every term contains either $x_{\alpha_1}$ or $x_{\alpha_8}$, while $\alpha_1$ and $\alpha_8$ are comparable to every other element of $L(P)$. 
  Thus, this polynomial cannot be reduced to an element of $\kk$ by $\calF_{L(P)}$. 
  This implies that $\calF_{L(P)}$ does not form a Khovanskii basis by Lemma~\ref{lemma:main lemma}. 
\end{example}

\begin{example}\label{example:divisor lattice of 2*3^m}
  Let $L(P)$ be a divisor lattice of $2 \cdot 3^m(m \geq 1)$.
  We label the elements of $L(P)$ as
  shown in Figure~\ref{fig:divisor lattice of 2*3^m}.
  \begin{figure}[ht]
    \centering
  {\scalebox{0.6}{
  \begin{tikzpicture}[line width=0.05cm]
  
  \coordinate (N11) at (-1,1); 
  \coordinate (N12) at (0,2); 
  \coordinate (N13) at (2,4); 
  
  \coordinate (N21) at (0,0); 
  \coordinate (N22) at (1,1); 
  \coordinate (N23) at (3,3); 
  \coordinate (N24) at (4,4); 
  \coordinate (N25) at (3,5);
  
  \draw  (N11)--(N12); 
  \draw  (N12)--(0.5,2.5); 
  \draw  (1.5,3.5)--(N13); 
  \draw[dotted]  (0.8,2.8)--(1.2,3.2);
  
  \draw  (N21)--(N22);
  \draw  (N22)--(1.5,1.5); 
  \draw  (2.5,2.5)--(N23);
  \draw  (N23)--(N24);
  \draw[dotted]  (1.8,1.8)--(2.2,2.2);
  \draw  (N24)--(N25);
  
  \draw  (N11)--(N21);
  \draw  (N12)--(N22);
  \draw  (N13)--(N23);
  \draw  (N13)--(N25);
  
  \draw [line width=0.05cm, fill=white] (N11) circle [radius=0.15] node [left=2pt] {\LARGE $\alpha_1$};
  \draw [line width=0.05cm, fill=white] (N12) circle [radius=0.15] node [left=2pt] {\LARGE $\alpha_2$};
  \draw [line width=0.05cm, fill=white] (N13) circle [radius=0.15] node [left=2pt] {\LARGE $\alpha_m$};
  \draw [line width=0.05cm, fill=white] (N25) circle [radius=0.15] node [left=2pt] {\LARGE $\alpha_{m+1}$};
  \draw [line width=0.05cm, fill=white] (N21) circle [radius=0.15] node [right=2pt] {\LARGE $\beta_1$};
  \draw [line width=0.05cm, fill=white] (N22) circle [radius=0.15] node [right=2pt] {\LARGE $\beta_2$};
  \draw [line width=0.05cm, fill=white] (N23) circle [radius=0.15] node [right=2pt] {\LARGE $\beta_m$};
  \draw [line width=0.05cm, fill=white] (N24) circle [radius=0.15] node [right=2pt] {\LARGE $\beta_{m+1}$};
  
  \end{tikzpicture}
  }}
  \caption{$L(P)$}
  \label{fig:divisor lattice of 2*3^m}
  \end{figure}
  
  Let $\calF_{L(P)} = \{x_{\alpha_i}x_{\beta_j} - x_{\beta_i}x_{\alpha_j} \mid 1 \leq i < j \leq m+1\}$.
  Then we observe that $\calF_{L(P)}$ is equal to the set of all $2 \times 2$-minors of a matrix of indeterminates 
  $\begin{pmatrix}
    x_{\alpha_1} & x_{\alpha_2} & \cdots & x_{\alpha_{m+1}} \\
    x_{\beta_1} & x_{\beta_2} & \cdots & x_{\beta_{m+1}}
    \end{pmatrix}$
  since \begin{equation*}
    f_{\alpha_i, \beta_j} = x_{\alpha_i}x_{\beta_j} - x_{\beta_i}x_{\alpha_j}
    = \det{\begin{pmatrix}
      x_{\alpha_i} & x_{\alpha_j} \\
      x_{\beta_i} & x_{\beta_j}
      \end{pmatrix}}. 
  \end{equation*}
  Therefore, we see that $\calR(L(P))$ is isomorphic to the \Pluecker algebra $A_{2,m+1}$. 
  Moreover, for a compatible monomial order $\preceq$, we have $\ini_{\preceq}(f_{\alpha_i, \beta_j})=x_{\alpha_i}x_{\beta_j}$, which is the product of diagonal entries of the submatrix. 
  Thus $\calF_{L(P)}$ is a Khovanskii basis of $\calR(L(P))$ with respect to a compatible monomial order by Proposition~\ref{proposition:plucker}. 
\end{example}

\section{generalized snake posets and $\{(2+2), (1+1+1)\}$-free posets}\label{section:generalized snake posets and (2+2) and (1+1+1)-free posets}
In this section,
we show a one-to-one correspondence between generalized snake posets 
and $\{(2+2),(1+1+1)\}$-free posets that are irreducible with respect to ordinal sum.

First, we introduce generalized snake posets. 
\begin{definition}[{\cite[Definition 3.1]{generalizedsnakeposet1}}]\label{def:generalized snake word}
  For $l \in \intnonneg$, a \textit{generalized snake word} is a word of the form 
  $\wb = w_0 w_1 \cdots w_l$, where $w_0 = \epsilon$ is the empty letter and $w_i$ is in the alphabet $\{L, R\}$ for $i = 1, \ldots, l$. 
\end{definition}
\begin{definition}[{\cite[Definition 3.2]{generalizedsnakeposet1}}]\label{def:generalized snake poset}
  Given a generalized snake word $\wb = w_0 w_1 \cdots w_l$, the \textit{generalized snake poset} $P(\wb)$ is recursively defined as follows:
  \begin{itemize}
    \item $P(w_0) = P(\epsilon)$ is the poset on elements 
    $\{\alpha_0, \alpha_1, \alpha_2, \alpha_3\}$ with cover relations
    $\alpha_0 \lessdot \alpha_1, \alpha_0 \lessdot \alpha_2, \alpha_1 \lessdot \alpha_3$
    and $\alpha_2 \lessdot \alpha_3$, where $\lessdot$ denotes the cover relation. 
    \item $P(w_0 w_1 \cdots w_l)$ is the poset
    $P(w_0 w_1 \cdots w_{l-1}) \cup \{\alpha_{2l+2}, \alpha_{2l+3}\}$
    with the added cover relations
    $\alpha_{2l+1} \lessdot \alpha_{2l+3}$, $\alpha_{2l+2} \lessdot \alpha_{2l+3}$, and 
    \begin{equation*}
      \left\{ \,
          \begin{aligned}
          & \alpha_{2l-1} \lessdot \alpha_{2l+2},
          \text{ if } l = 1 \text{ and } w_l = L,
          \text{ or } l \geq 2 \text{ and } w_{l-1}w_l \in \{RL, LR\}, \\
          & \alpha_{2l} \lessdot \alpha_{2l+2},
          \text{ if } l = 1 \text{ and } w_l = R,
          \text{ or } l \geq 2 \text{ and } w_{l-1}w_l \in \{LL, RR\}.
          \end{aligned}
      \right.
      \end{equation*}
  \end{itemize}
\end{definition}
  If $w = w_0w_1\cdots w_l$ is a generalized snake word, then $P(w)$ is a distributive lattice of width two and rank $l + 2$. 
  Note that although we refer to~\cite[Definition 3.2]{generalizedsnakeposet1} for the definition of generalized snake posets,  
  we employ the definition of generalized snake posets with upside down version for the convenience in our discussions. 

\begin{example}\label{example:small examples of generalized snake poset}
  The generalized snake poset $P(\epsilon LLRL)$ is the distributive
  lattice depicted in Figure~\ref{fig:small example of generalized snake poset}.
  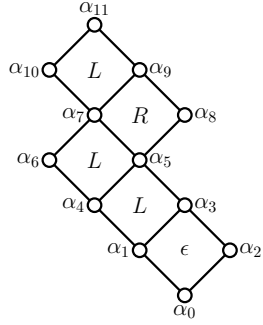
\begin{figure}[ht]
  \centering
  {\scalebox{0.60}{
    \begin{tikzpicture}[line width=0.05cm]

      \coordinate (N0) at (0,0);   
      \coordinate (N1) at (-1,1);  
      \coordinate (N2) at (1,1);   
      \coordinate (N3) at (0,2);   
      \coordinate (N4) at (-2,2);  
      \coordinate (N5) at (-1,3);  
      \coordinate (N6) at (-3,3);  
      \coordinate (N7) at (-2,4);  
      \coordinate (N8) at (0,4);   
      \coordinate (N9) at (-1,5);  
      \coordinate (N10) at (-3,5); 
      \coordinate (N11) at (-2,6); 
    
      \draw (N0)--(N1); \draw (N0)--(N2); \draw (N1)--(N3); \draw (N2)--(N3); 
      \draw (N1)--(N4); \draw (N4)--(N5); \draw (N3)--(N5); 
      \draw (N4)--(N6); \draw (N6)--(N7); \draw (N5)--(N7); 
      \draw (N5)--(N8); \draw (N8)--(N9); \draw (N7)--(N9); 
      \draw (N7)--(N10); \draw (N9)--(N11); \draw (N10)--(N11); 
    
      \draw [line width=0.05cm, fill=white] (N0) circle [radius=0.15] node [below=2pt] {\LARGE $\alpha_0$};
      \draw [line width=0.05cm, fill=white] (N1) circle [radius=0.15] node [left=2pt] {\LARGE $\alpha_1$};
      \draw [line width=0.05cm, fill=white] (N2) circle [radius=0.15] node [right=2pt] {\LARGE $\alpha_2$};
      \draw [line width=0.05cm, fill=white] (N3) circle [radius=0.15] node [right=2pt] {\LARGE $\alpha_3$};
      \draw [line width=0.05cm, fill=white] (N4) circle [radius=0.15] node [left=2pt] {\LARGE $\alpha_4$};
      \draw [line width=0.05cm, fill=white] (N5) circle [radius=0.15] node [right=2pt] {\LARGE $\alpha_5$};
      \draw [line width=0.05cm, fill=white] (N6) circle [radius=0.15] node [left=2pt] {\LARGE $\alpha_6$};
      \draw [line width=0.05cm, fill=white] (N7) circle [radius=0.15] node [left=2pt] {\LARGE $\alpha_7$};
      \draw [line width=0.05cm, fill=white] (N8) circle [radius=0.15] node [right=2pt] {\LARGE $\alpha_8$};
      \draw [line width=0.05cm, fill=white] (N9) circle [radius=0.15] node [right=2pt] {\LARGE $\alpha_9$};
      \draw [line width=0.05cm, fill=white] (N10) circle [radius=0.15] node [left=2pt] {\LARGE $\alpha_{10}$};
      \draw [line width=0.05cm, fill=white] (N11) circle [radius=0.15] node [above=2pt] {\LARGE $\alpha_{11}$};
    
      \node at (0,1) {\LARGE $\epsilon$}; 
      \node at (-1,2) {\LARGE $L$};      
      \node at (-2,3) {\LARGE $L$};      
      \node at (-1,4) {\LARGE $R$};
      \node at (-2,5) {\LARGE $L$};
    \end{tikzpicture}
  }}
  \caption{$P(\epsilon LLRL)$}
  \label{fig:small example of generalized snake poset}
\end{figure}
\end{example}

There are four possible shapes of generalized snake posets $P(w_0w_1 \cdots w_l)$.
These correspond to the following four cases:
\begin{itemize}
  \item $w_1 = L, w_l = R$,
  \item $w_1 = L, w_l = L$,
  \item $w_1 = R, w_l = R$,
  \item $w_1 = R, w_l = L$.
\end{itemize}
For example, given positive integers $s_1, \ldots, s_k, t_1, \ldots, t_k$ with $k \geq 1$, the generalized snake poset
$P(\epsilon L^{s_1}R^{t_1} \cdots L^{s_k}R^{t_k})$ is the distributive lattice depicted in Figure~\ref{fig:main lattice}. 
We omit the description of the other cases.
\begin{figure}[ht]
  \centering
  {\scalebox{0.70}{
    \begin{tikzpicture}[line width=0.05cm]
    
      \coordinate (N11) at (3,1); 
      \coordinate (N12) at (2,2); 
      \coordinate (N13) at (0,4); 
      \coordinate (N14) at (1,5);
      \coordinate (N15) at (3,7);
      \coordinate (N16) at (4,8); 
      \coordinate (N17) at (3,9);
      \coordinate (N19) at (-1,11);
      \coordinate (N110) at (0,12); 
      \coordinate (N111) at (2,14);
      \coordinate (N112) at (3,15);
      
      \coordinate (N21) at (2,0); 
      \coordinate (N22) at (1,1); 
      \coordinate (N23) at (-1,3); 
      \coordinate (N24) at (-2,4); 
      \coordinate (N25) at (-1,5);
      \coordinate (N26) at (0,6); 
      \coordinate (N27) at (2,8);
      \coordinate (N29) at (-2,12);
      \coordinate (N210) at (-1,13); 
      \coordinate (N211) at (1,15);
      \coordinate (N212) at (2,16); 
      
      \draw  (N11)--(N12); 
      \draw  (N12)--(1.5,2.5); 
      \draw  (0.5,3.5)--(N13); 
      \draw[dotted]  (1.2,2.8)--(0.8,3.2);
      \draw  (N13)--(N14);
      \draw  (N14)--(1.5,5.5);
      \draw  (2.5,6.5)--(N15);
      \draw  (N15)--(N16);
      \draw  (N16)--(N17);
      \draw[dotted]  (1.8,5.8)--(2.2,6.2);
      \draw  (N17)--(2.5,9.5);
      \draw  (-0.5,10.5)--(N19);
      \draw  (N19)--(N110);
      \draw  (N110)--(0.5,11.5);
      \draw  (N110)--(0.5,12.5);
      \draw  (1.5,13.5)--(N111);
      \draw  (N111)--(N112);
      \draw[dotted]  (0.8,12.8)--(1.2,13.2);
      
      \draw  (N21)--(N22);
      \draw  (N22)--(0.5,1.5); 
      \draw  (-0.5,2.5)--(N23);
      \draw  (N23)--(N24);
      \draw[dotted]  (0.2,1.8)--(-0.2,2.2);
      \draw  (N24)--(N25);
      \draw[dotted]  (0.8,6.8)--(1.2,7.2);
      \draw  (N25)--(N26);
      \draw  (N26)--(0.5,6.5);
      \draw  (1.5,7.5)--(N27);
      \draw  (N27)--(1.5,8.5);
      \draw  (N29)--(N210);
      \draw  (N210)--(-0.5,13.5);
      \draw  (0.5,14.5)--(N211);
      \draw  (N211)--(N212);
      \draw[dotted]  (-0.2,13.8)--(0.2,14.2);
      
      \draw  (N11)--(N21);
      \draw  (N12)--(N22);
      \draw  (N13)--(N23);
      \draw  (N13)--(N25);
      \draw  (N14)--(N26);
      \draw  (N15)--(N27);
      \draw  (N17)--(N27);
      \draw  (N19)--(N29);
      \draw  (N110)--(N210);
      \draw  (N111)--(N211);
      \draw  (N112)--(N212);
      \draw[dotted]  (1,9.8)--(1,10.2);
      
      \draw [line width=0.05cm, fill=white] (N11) circle [radius=0.15]; 
      \draw [line width=0.05cm, fill=white] (N12) circle [radius=0.15]; 
      \draw [line width=0.05cm, fill=white] (N13) circle [radius=0.15]; 
      \draw [line width=0.05cm, fill=white] (N14) circle [radius=0.15]; 
      \draw [line width=0.05cm, fill=white] (N15) circle [radius=0.15]; 
      \draw [line width=0.05cm, fill=white] (N16) circle [radius=0.15]; 
      \draw [line width=0.05cm, fill=white] (N17) circle [radius=0.15]; 
      \draw [line width=0.05cm, fill=white] (N19) circle [radius=0.15]; 
      \draw [line width=0.05cm, fill=white] (N110) circle [radius=0.15]; 
      \draw [line width=0.05cm, fill=white] (N111) circle [radius=0.15]; 
      \draw [line width=0.05cm, fill=white] (N112) circle [radius=0.15]; 
      \draw [line width=0.05cm, fill=white] (N21) circle [radius=0.15]; 
      \draw [line width=0.05cm, fill=white] (N22) circle [radius=0.15]; 
      \draw [line width=0.05cm, fill=white] (N23) circle [radius=0.15]; 
      \draw [line width=0.05cm, fill=white] (N24) circle [radius=0.15]; 
      \draw [line width=0.05cm, fill=white] (N25) circle [radius=0.15]; 
      \draw [line width=0.05cm, fill=white] (N26) circle [radius=0.15]; 
      \draw [line width=0.05cm, fill=white] (N27) circle [radius=0.15]; 
      \draw [line width=0.05cm, fill=white] (N29) circle [radius=0.15]; 
      \draw [line width=0.05cm, fill=white] (N210) circle [radius=0.15]; 
      \draw [line width=0.05cm, fill=white] (N211) circle [radius=0.15]; 
      \draw [line width=0.05cm, fill=white] (N212) circle [radius=0.15]; 
      
      \draw[<-] (2.5,0.5)--(1,2);
      \draw[->] (0,3)--(-1.5,4.5);
      \node at (0.5,2.5) {\LARGE $s_1+1$};
      
      \draw[<-] (3.5,8.5)--(2,7);
      \draw[<-] (-1.5,3.5)--(1,6);
      \node at (1.5,6.5) {\LARGE $t_1+1$};
      
      \draw[<-] (-1.5,11.5)--(0,13);
      \draw[<-] (2.5,15.5)--(1,14);
      \node at (0.5,13.5) {\LARGE $t_k+1$};
      
  \end{tikzpicture}
  }}
  \caption{$P(\epsilon L^{s_1}R^{t_1} \cdots L^{s_k}R^{t_k})$}
  \label{fig:main lattice}
\end{figure}
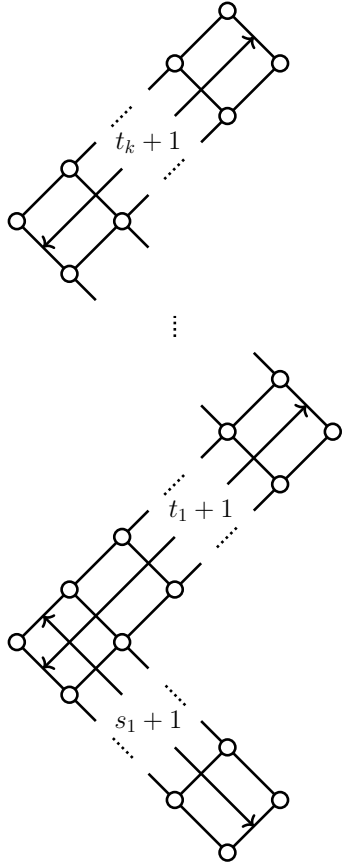

For a poset $P$ and an element $a \in P$, let $D(a) \coloneqq \{b \in P \mid b \lneq a\}$ (resp. $U(a) \coloneqq \{b \in P \mid b \gneq a\}$) denote the \textit{downset} (resp. \textit{upset}) of $a$.
It is known that a poset $P$ is $(2+2)$-free if and only if the set of downsets $\{D(x) \mid x \in P\}$ 
can be linearly ordered by inclusion~\cite{BogartFreePoset}. 
Furthermore, Dukes, Jel\'{i}nek, and Kubitzke ~\cite{DukesVitKubitzkeCompositionmatrix} found a one-to-one correspondence 
between $(2+2)$-free posets and composition matrices. 
Here, we call an upper triangular matrix $M$ a \textit{composition matrix} on the set $[n] = \{1, \ldots, n\}$ if 
the entries of $M$ are subsets of $[n]$, where all the nonempty entries form a partition of $[n]$, 
and there are no rows or columns that contain only empty sets. 

\begin{definition}[{\cite[Definition 1]{DukesVitKubitzkeCompositionmatrix}}]
  For a $(2+2)$-free poset $P$ on $[n]$, let
  $\emptyset = D_0 \subsetneq D_1 \subsetneq \ldots \subsetneq D_m$ be the sequence of downsets of $P$. 
  Let $L_i \coloneqq \{a \in P \mid D(a) = D_i\}$ for all $0 \leq i \leq m$.
  Let $D_{m+1} \coloneqq [n]$ and $K_j \coloneqq D_{j+1} \setminus D_j$ for all $0 \leq j \leq m$.
  We define a matrix $M = \Gamma(P)$ as the $(m+1) \times (m+1)$-matrix over the powerset of $[n]$, 
  where the $(i,j)$-entry $M_{ij}$ of $M$ is given as $L_{i-1} \cap K_{j-1}$ for all $1 \leq i,j \leq m+1$. 
\end{definition}

It is proved in~\cite[Theorem 9]{DukesVitKubitzkeCompositionmatrix} that 
$\Gamma$ is a bijection between the set of $(2+2)$-free posets on $[n]$ and the set of composition matrices on $[n]$. 
Moreover, we can obtain information about $P$ from the composition matrix $\Gamma(P)$. 
\begin{lemma}[{\cite[Lemma 8]{DukesVitKubitzkeCompositionmatrix}}]\label{lemma:information from composition matrix}
  Assume that a $(2+2)$-free poset $P$ corresponds to a composition matrix $M=\Gamma(P)$ which is a $(m+1) \times (m+1)$-matrix. 
  Let $R_i \subset [n]$ (resp. $C_j \subset [n]$) be the union of the cells in the $i$-th row (resp. the $j$-th column) of $M$. 
  Then the following assertions hold:  
  \begin{itemize}
    \item For $a \in M_{ij}$, we have $D(a) = \bigcup_{k=1}^{i-1}C_k$ and $U(a) = \bigcup_{k=j+1}^{m+1}R_k$.
    \item For two elements $a, b \in P$, we have $D(a) = D(b)$ if and only if $a$ and $b$ appear in the same row of $M$. 
    \item For two elements $a, b \in P$, we have $U(a) = U(b)$ if and only if $a$ and $b$ appear in the same column of $M$. 
  \end{itemize}
\end{lemma}

\begin{example}\label{example:example of 2+2-free poset}
  Let $P$ be a $(2+2)$-free poset on $[6]=\{1,\ldots,6\}$ depicted in Figure~\ref{fig:example of 2+2-free poset}.
  Then the sequences of downsets, that of $L_i$'s, and that of $K_j$'s are as follows: 
  \begin{align*}
    &D_0 = \emptyset, \; D_1 = \{1\}, \; D_2 = \{1,2\}, \; D_3 = \{1,2,5\}, \; D_4 = \{1,2,3,5\}, D_5 = [6]; \\
    &L_0 = \{1,5\}, \; L_1 = \{2\}, \; L_2 = \{3\}, \; L_3 = \{6\}, \; L_4 = \{4\}; \\
    &K_0 = \{1\}, \; K_1 = \{2\}, \; K_2 = \{5\}, \; K_3 = \{3\}, \; K_4 = \{4,6\}.
  \end{align*}
  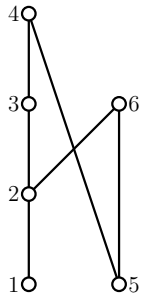
\begin{figure}
  \centering
  {\scalebox{0.6}{
  \begin{tikzpicture}[line width=0.05cm]
  
  \coordinate (a0) at (0,0); 
  \coordinate (a1) at (0,2); 
  \coordinate (a2) at (0,4); 
  \coordinate (a3) at (0,6);
  \coordinate (b0) at (2,0); 
  \coordinate (b1) at (2,4);
  
  \draw  (a0)--(a3);
  \draw  (b0)--(b1);
  \draw  (a3)--(b0); 
  \draw  (a1)--(b1); 

  \draw [line width=0.05cm, fill=white] (a0) circle [radius=0.15] node [left=2pt] {\LARGE $1$};
  \draw [line width=0.05cm, fill=white] (a1) circle [radius=0.15] node [left=2pt] {\LARGE $2$};
  \draw [line width=0.05cm, fill=white] (a2) circle [radius=0.15] node [left=2pt] {\LARGE $3$};
  \draw [line width=0.05cm, fill=white] (a3) circle [radius=0.15] node [left=2pt] {\LARGE $4$};
  \draw [line width=0.05cm, fill=white] (b0) circle [radius=0.15] node [right=2pt] {\LARGE $5$};
  \draw [line width=0.05cm, fill=white] (b1) circle [radius=0.15] node [right=2pt] {\LARGE $6$};
  
  \end{tikzpicture}
  }}
  \caption{An example of a $\{(2+2),(1+1+1)\}$-free poset}
  \label{fig:example of 2+2-free poset}
  \end{figure}
  This poset corresponds to the following composition matrix $M = \Gamma(P)$.
  \begin{equation*}
    M =
    \begin{pmatrix*}
      \{1\} & \emptyset & \{5\} & \emptyset & \emptyset \\
      \emptyset & \{2\} & \emptyset & \emptyset & \emptyset \\
      \emptyset & \emptyset & \emptyset & \{3\} & \emptyset \\
      \emptyset & \emptyset & \emptyset & \emptyset & \{6\} \\
      \emptyset & \emptyset & \emptyset & \emptyset & \{4\}
    \end{pmatrix*}
  \end{equation*}
\end{example}

The poset $P$ described in Example~\ref{example:example of 2+2-free poset}
is a $\{(2+2),(1+1+1)\}$-free poset on $[6]$ which is irreducible with respect to ordinal sum. 
The size of the composition matrix $M = \Gamma(P)$ is $5 = 6 - 1$. 
Furthermore, each row contains exactly one singleton except for the first row, where the union of its entries forms a set of two elements.
Similarly, each column contains exactly one singleton except for the last column, where the union of its entries forms a set of two elements. 
These observations hold for any $\{(2+2),(1+1+1)\}$-free poset that is irreducible with respect to ordinal sum. 

In what follows, we omit the case where $P$ is a singleton. 
\begin{lemma}\label{lemma:downsets of (2+2) and (1+1+1)-free poset}
  Let $P$ be a $\{(2+2),(1+1+1)\}$-free poset on $[n]$ that is irreducible with respect to ordinal sum. 
  Then the following assertions hold.
  \begin{itemize}
    \item The number of downsets, except for $\emptyset$ and $P$, is $n-2$.
    \item $L_0$ consists of two elements, while each $L_i$ for $1 \leq i \leq n-2$ is a singleton.
    \item $K_{n-2}$ consists of two elements, while each $K_j$ for $0 \leq i \leq n-3$ is a singleton.
  \end{itemize}
\end{lemma}
\begin{proof}
  Let $M = \Gamma(P)$ be the composition matrix corresponding to $P$.
  We show the following: 
  \begin{itemize}
    \item $M$ is a $(n-1) \times (n-1)$-matrix.
    \item Each row contains exactly one singleton except for the first row,
    where the union of its entries $R_1$ consists of two elements.
    \item Each column contains exactly one singleton except for the last column,
    where the union of its entries $C_{n-1}$ consists of two elements.
  \end{itemize}
  By the construction of $M$, the first claim implies that the number of downsets is $n-2$. 
  Furthermore, the second and third claims imply the corresponding claims on $P$, such as $R_i =L_{i-1}, C_j = K_{j-1}$. 
  These follow from the fact that $M_{ij} = L_{i-1} \cap K_{j-1}$ and the $L_i$'s and $K_j$'s form partitions of $[n]$. 

  Since $P$ is $(1+1+1)$-free, irreducible with respect to ordinal sum, and is not a singleton, the width of $P$ is $2$. 
  Therefore, $P$ can be expressed as the disjoint union of two chains \begin{equation*}
    P = \{a_0 \lessdot a_1 \lessdot \cdots \lessdot a_s\} \sqcup \{b_0 \lessdot b_1 \lessdot \cdots \lessdot b_t\},
  \end{equation*}
  where $s+t+2 = n$. 

  If $D(a) = D(b)$, then $a$ and $b$ are incomparable, 
  so we can assume $a = a_i$ for some $i$ and $b = b_j$ for some $j$. 
  Thus, $P$ can be written as an ordinal sum of $D(a_i) = D(b_j) = \{a_0, \ldots, a_{i-1}, b_0, \ldots, b_{j-1}\}$ 
  and $\{a_i, \ldots, a_s, b_j, \ldots, b_t\}$, a contradiction unless the case with $a_i = a_0$ and $b_j = b_0$, i.e., $D(a_0) = D(b_0) = \emptyset$. 
  Similarly, if $U(a) = U(b)$, then $\{a, b\}$ must be equal to $\{a_s, b_t\}$ and $U(a_s) = U(b_t) = \emptyset$. 

  Since the downsets $D(p)$ are distinct for each $p \in P$ except for $D(a_0) = D(b_0) = \emptyset$, 
  the number of downsets is $n-2$. Thus, $M$ is a $(n-1) \times (n-1)$-matrix. 
  By the definition of composition matrices, each row contains exactly one singleton except for exactly one row. 
  For the exceptional row, the union of its entries is $\{a_0, b_0\}$ by the second claim of Lemma~\ref{lemma:information from composition matrix}.
  Moreover, since $D(a_0)=D(b_0)=\emptyset$, we identify the exceptional row as the first row of $M$
  by the first claim of Lemma~\ref{lemma:information from composition matrix}.
  Similarly, each column contains exactly one singleton except for the
  last column, where the union of its entries is $C_{n-1} = \{a_s, b_t\}$.
\end{proof}

\begin{example}
  For the poset described in Example~\ref{example:example of 2+2-free poset},
  the sequence of downsets of $P$ is given by
  \begin{equation*}
    \emptyset = D(1) = D(5) \subsetneq D(2) \subsetneq D(3) \subsetneq D(6) \subsetneq D(4) \subsetneq [6].
  \end{equation*}
  where each difference is a singleton, except for $[6] \setminus D(4) = \{4,6\}$.
\end{example}

For $s, t \geq 0$, let
\begin{equation*}
  P = \{a_0 \lessdot a_1 \lessdot \cdots \lessdot a_s\} \sqcup \{b_0 \lessdot b_1 \lessdot \cdots \lessdot b_t\}
\end{equation*}
be a  $\{(2+2),(1+1+1)\}$-free poset that is irreducible with respect to ordinal sum.
There are four possible shapes of $\{(2+2),(1+1+1)\}$-free posets 
that is irreducible with respect to ordinal sum.
These correspond to the following four cases:
\begin{itemize}
  \item The smallest nonempty downset is $D(a_1)$, and the largest proper downset is $D(b_t)$,
  \item The smallest nonempty downset is $D(a_1)$, and the largest proper downset is $D(a_s)$,
  \item The smallest nonempty downset is $D(b_1)$, and the largest proper downset is $D(b_t)$,
  \item The smallest nonempty downset is $D(b_1)$, and the largest proper downset is $D(a_s)$.
\end{itemize}
For example, if $P$ is in the first case, there exist positive integers
$s_1, \ldots, s_k, t_1, \ldots, t_k$ with $k \geq 1$
such that the sequence of its downsets is 
\begin{align*}
  \emptyset = D(a_0) = D(b_0) &\subsetneq D(a_1) \subsetneq \ldots \subsetneq D(a_{s_1}) \\
  &\subsetneq D(b_1) \subsetneq \ldots \subsetneq D(b_{t_1}) \\
  & \hspace{55pt} \vdots \\
  &\subsetneq D(b_{t_1+\cdots+t_{k-1}+1}) \subsetneq \ldots \subsetneq D(b_{t_1+\cdots+t_k}) \subsetneq P
\end{align*}
and each difference is a singleton, except for $P \setminus D(b_{t_1+\cdots+t_k}) = \{a_s, b_t\}$.
Then the Hasse diagram of $P$ becomes like as depicted in Figure~\ref{fig:main poset}.
We omit the description of the other cases.
\begin{figure}[ht]
  \centering
  {\scalebox{0.70}{
  \begin{tikzpicture}[line width=0.05cm]
    \coordinate (N11) at (2,0); 
    \coordinate (N12) at (2,3); 
    \coordinate (N13) at (2,5); 
    \coordinate (N14) at (2,7); 
    \coordinate (N15) at (2,9);
    \coordinate (N16) at (2,12); 
    \coordinate (N17) at (2,14);
    
    \coordinate (N21) at (0,0); 
    \coordinate (N22) at (0,2); 
    \coordinate (N23) at (0,4); 
    \coordinate (N24) at (0,6); 
    \coordinate (N25) at (0,8);
    \coordinate (N26) at (0,11); 
    \coordinate (N27) at (0,14);
    
    \draw  (N11)--(N12); 
    \draw  (N12)--(2,3.5); 
    \draw  (2,4.5)--(N13); 
    \draw[dotted]  (2,3.8)--(2,4.2);
    \draw  (N13)--(N14);
    \draw  (N14)--(N15);
    \draw  (N15)--(2,9.5);
    \draw  (2,10.5)--(N16);
    \draw  (N16)--(2,12.5);
    \draw  (2,13.5)--(N17);
    \draw[dotted]  (2,12.8)--(2,13.2);
    
    \draw  (N21)--(0,0.5);
    \draw  (0,1.5)--(N22); 
    \draw  (N22)--(N23);
    \draw  (N23)--(N24);
    \draw[dotted]  (0,0.8)--(0,1.2);
    \draw  (N24)--(0,6.5);
    \draw[dotted]  (0,6.8)--(0,7.2);
    \draw  (0,7.5)--(N25);
    \draw  (N25)--(0,9.5);
    \draw  (0,10.5)--(N26);
    \draw  (N26)--(N27);
    
    \draw  (N12)--(N22);
    \draw  (N13)--(N24);
    \draw  (N15)--(N25);
    \draw  (N16)--(N26);
    \draw[dotted]  (1,9.8)--(1,10.2);
    
    \draw [line width=0.05cm, fill=white] (N11) circle [radius=0.15] node[right=2pt] {\LARGE $b_0$};
    \draw [line width=0.05cm, fill=white] (N12) circle [radius=0.15] node[right=2pt] {\LARGE $b_1$};
    \draw [line width=0.05cm, fill=white] (N13) circle [radius=0.15] node[right=2pt] {\LARGE $b_{t_1-1}$};
    \draw [line width=0.05cm, fill=white] (N14) circle [radius=0.15] node[right=2pt] {\LARGE $b_{t_1}$};
    \draw [line width=0.05cm, fill=white] (N15) circle [radius=0.15] node[right=2pt] {\LARGE $b_{t_1+1}$};
    \draw [line width=0.05cm, fill=white] (N16) circle [radius=0.15] node[right=2pt] {\LARGE $b_{t_1 + \cdots + t_{k-1}+1}$};
    \draw [line width=0.05cm, fill=white] (N17) circle [radius=0.15] node[right=2pt] {\LARGE $b_{t_1 + \cdots + t_k}$};
    \draw [line width=0.05cm, fill=white] (N21) circle [radius=0.15] node[left=2pt] {\LARGE $a_0$};
    \draw [line width=0.05cm, fill=white] (N22) circle [radius=0.15] node[left=2pt] {\LARGE $a_{s_1-1}$};
    \draw [line width=0.05cm, fill=white] (N23) circle [radius=0.15] node[left=2pt] {\LARGE $a_{s_1}$};
    \draw [line width=0.05cm, fill=white] (N24) circle [radius=0.15] node[left=2pt] {\LARGE $a_{s_1+1}$};
    \draw [line width=0.05cm, fill=white] (N25) circle [radius=0.15] node[left=2pt] {\LARGE $a_{s_1+s_2-1}$}; 
    \draw [line width=0.05cm, fill=white] (N26) circle [radius=0.15] node[left=2pt] {\LARGE $a_{s_1 + \cdots + s_k-1}$}; 
    \draw [line width=0.05cm, fill=white] (N27) circle [radius=0.15] node[left=2pt] {\LARGE $a_{s_1 + \cdots + s_k}$};
    
    \node at (0,15) {$\;$}; 
    \node at (0,-1) {$\;$};    
  \end{tikzpicture}
  }}
  \caption{A structure of a $\{(2+2),(1+1+1)\}$-free poset}
  \label{fig:main poset}
\end{figure}
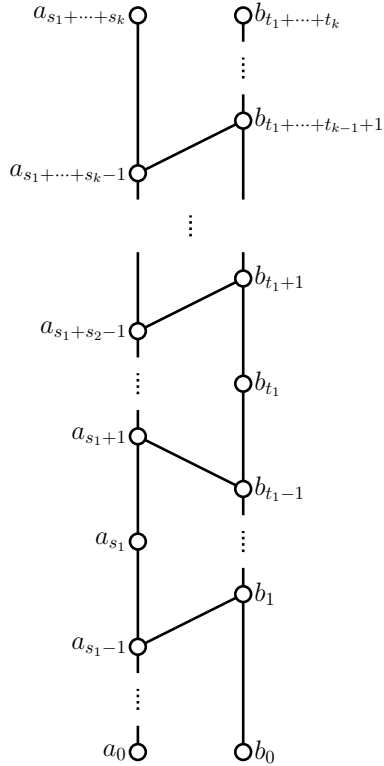

\begin{proposition}\label{proposition:correspondence free poset and generalized snake poset}
  Let $P$ be a poset and $L(P)$ the distributive lattice corresponding to $P$. 
  Then the following conditions are equivalent:
  \begin{enumerate}
    \item[{\em (2)}] $P$ is a $\{(2+2),(1+1+1)\}$-free poset that is irreducible with respect to ordinal sum.  
    \item[{\em (3)}] $L(P)$ is a generalized snake poset. 
  \end{enumerate}
\end{proposition}
\begin{proof}
  \textbf{(2) $\implies$ (3)}:
  For $s, t \in \intnonneg$, let
  \begin{equation*}
    P = \{a_0 \lessdot a_1 \lessdot \cdots \lessdot a_s\} \sqcup \{b_0 \lessdot b_1 \lessdot \cdots \lessdot b_t\}
  \end{equation*}
  be a $\{(2+2),(1+1+1)\}$-free poset that is irreducible with respect to ordinal sum.

  If $s = t = 0$, then $P = \{a_0, b_0\}$ and they are incomparable, so $L(P)$ is a boolean lattice of rank 2. 
  In particular, $L(P)$ consists of the four elements: 
  $\alpha_0 = \emptyset, \alpha_1 = \{a_0\}, \alpha_2 = \{b_0\}$,
  and $\alpha_3 = \{a_0, b_0\}$. It is the generalized snake poset $P(\epsilon)$.

  Now, we show that $L(P)$ is the generalized snake poset $P(w_0w_1 \cdots w_l)$
  by induction on $l = s + t \geq 1$.

  If $s=1, t=0$, then $P$ is a union of $\{a_0, b_0\}$ and a new element $\{a_1\}$.
  Furthermore, $L(P)$ is a union of $\{\alpha_0, \alpha_1, \alpha_2, \alpha_3\}$ and
  $\{\alpha_4=\{a_0, a_1\}, \alpha_5=\{a_0, a_1, b_0\}\}$.
  In this case, the largest downset of $P$ is $D(a_1)$.
  Therefore, $\alpha_4$ and $\alpha_5$ are expressed as $D(a_1) \cup \{a_1\} = P \setminus \{b_0\}$
  and $P$, respectively.
  There are inclusions $\alpha_3 \subset \alpha_5, \alpha_4 \subset \alpha_5$,
  and $\alpha_1 \subset \alpha_4$.
  They are cover relations in $L(P)$ because the differences of the sets are singletons.
  Thus, $L(P)$ is the generalized snake poset $P(\epsilon L)$.

  Similarly, if $s=0, t=1$, then $P$ is a union of $\{a_0, b_0\}$ and a new element $\{b_1\}$.
  In this case, the largest downset of $P$ is $D(b_1)$.
  Furthermore, $L(P)$ is a union of $\{\alpha_0, \alpha_1, \alpha_2, \alpha_3\}$
  and $\{\alpha_4=D(b_1) \cup \{b_1\} = P \setminus \{a_0\}, \alpha_5=P\}$.
  Thus, $L(P)$ is the generalized snake poset $P(\epsilon R)$.

  For $l = s+t \geq 1$, we assume $L(P) = P(w_0w_1 \cdots w_l)$ with $w_l = L$.
  Also, we assume that the largest downset of $P$ is $D(a_s)$ and elements
  of $L(P)$ are labeled with 
  $\alpha_{2l+1} = P \setminus \{a_s\}, \alpha_{2l+2} = D(a_s) \cup \{a_s\} = P \setminus \{b_t\}$
  and $\alpha_{2l+3} = P$.
  If we add $a_{s+1}$ to $P$ and $P' \coloneqq P \cup \{a_{s+1}\}$
  is a $\{(2+2),(1+1+1)\}$-free poset that is irreducible with respect to ordinal sum, 
  then the largest downset of $P'$ becomes $D(a_{s+1})$. 
  Furthermore, we have $D(a_s) \subset D(a_{s+1}) \subset P'$ and their differences are
  $D(a_{s+1}) \setminus D(a_s) = \{a_s\}, P' \setminus D(a_{s+1}) = \{a_{s+1}, b_t\}$.
  These facts are based on Lemma~\ref{lemma:downsets of (2+2) and (1+1+1)-free poset}.
  Then $L(P')$ is a union of $L(P)$ and $\{\alpha_{2l+4} = D(a_{s+1}) \cup \{a_{s+1}\} = P' \setminus \{b_t\}, \alpha_{2l+5} = P'\}$. 
  There are inclusions $\alpha_{2l+3} \subset \alpha_{2l+5}, \alpha_{2l+4} \subset \alpha_{2l+5}$, and $\alpha_{2l+2} \subset \alpha_{2l+4}$, 
  which are cover relations in $L(P)$ because the differences of the sets are singletons. 
  Thus, $L(P')$ is the generalized snake poset $P(w_0 w_1 \cdots w_lw_{l+1})$
  with $w_lw_{l+1} = LL$.
  On the other case that we add $b_{t+1}$ to $P$ and $P' \coloneqq P \cup \{b_{t+1}\}$ is a $\{(2+2),(1+1+1)\}$-free poset 
  that is irreducible with respect to ordinal sum, the largest downset of $P'$ becomes $D(b_{t+1})$. 
  Furthermore, we have $D(a_s) \subset D(b_{t+1}) \subset P'$ and their differences are 
  $D(b_{t+1}) \setminus D(a_s) = \{b_t\}, P' \setminus D(b_{t+1}) = \{a_s, b_{t+1}\}$. 
  Then $L(P')$ is a union of $L(P)$ and $\{\alpha_{2l+4} = D(b_{t+1}) \cup \{b_{t+1}\} = P' \setminus \{a_s\}, \alpha_{2l+5} = P'\}$. 
  There are inclusions $\alpha_{2l+3} \subset \alpha_{2l+5}, \alpha_{2l+4} \subset \alpha_{2l+5}$, and $\alpha_{2l+1} \subset \alpha_{2l+4}$, 
  which are cover relations in $L(P)$ because the differences of the sets are singletons. 
  Thus, $L(P')$ is the generalized snake poset $P(w_0 w_1 \cdots w_lw_{l+1})$ with $w_lw_{l+1} = LR$.

  Similarly, we can prove the case that the largest downset of $P$ is $D(b_t)$.
  Thus, $L(P)$ is a generalized snake poset for all $s, t \in \intnonneg$.

\smallskip

\noindent\textbf{(3) $\implies$ (2)}:
  If $P$ contains a poset depicted in either
  Figure~\ref{fig:forbidden subposet 1} or Figure~\ref{fig:forbidden subposet 2}
  as a subposet,
  then $L(P)$ contains a boolean lattice of rank 3 or a divisor lattice of 36 as a sublattice. 
  The width of each of the two lattices is 3, whereas the width of generalized snake posets is 2.
  Therefore, $L(P)$ is not a generalized snake poset. 
  Moreover, if $P$ can be expressed as the ordinal sum of two nonempty subposets $P_1$ and $P_2$, 
  then $L(P)$ is the union of $L(P_1)$ and $L(P_2)$ such that 
  the minimal element of $L(P_1)$ is equal to the maximal element of $L(P_2)$. 
  Therefore, there is an element of $L(P)$ other than maximal and minimal elements that is comparable to every other element of $L(P)$. 
  However, generalized snake posets never contain such element. 
\end{proof}

\begin{remark}
  The poset described in Example~\ref{example:example of 2+2-free poset} corresponds to
  the distributive lattice described in Example~\ref{example:small examples of generalized snake poset}.
  The poset depicted in Figure~\ref{fig:main poset} corresponds to the distributive lattice
  depicted in Figure~\ref{fig:main lattice}.
\end{remark}

Now, we are ready to give a proof of Theorem~\ref{theorem:main theorem}. 
Since we already know $(2) \iff (3)$ by Proposition~\ref{proposition:correspondence free poset and generalized snake poset},
it suffices to prove $(1) \implies (2)$ and $(3) \implies (1)$. 

\noindent\textbf{(1) $\implies$ (2)}:
  We prove the contraposition. If $P$ contains a poset depicted in either
  Figure~\ref{fig:forbidden subposet 1} or Figure~\ref{fig:forbidden subposet 2} as a subposet,
  then $L(P)$ contains a divisor lattice of 36 or a boolean lattice of rank 3 as a sublattice. 
  As seen in Examples~\ref{example:dvisor lattice of 36} and \ref{example:boolean lattice of rank 3}, 
  the set of polynomials arising from either of these lattices does not form a Khovanskii basis. 
  Thus, by Lemma~\ref{lemma:main lemma}, $\calF_{L(P)}$ is not a Khovanskii basis of $\calR(L(P))$. 

\noindent\textbf{(3) $\implies$ (1)}:
  If $L(P)$ is a generalized snake poset, since $L(P)$ is a disjoint union of two chains, the maximal element and the minimal element, 
  the co-comparability graph $G_{L(P)}$ becomes a bipartite graph. 
  Then we apply Lemma~\ref{lemma:main lemma} to $L(P)$. By Lemma~\ref{corollary:toric ideal of bipartite graph},
  every primitive even closed walk of $G_{L(P)}$ is an even cycle of $G_{L(P)}$. 
  Let $C_1, \ldots, C_t$ be even cycles of $G_{L(P)}$, 
  and $L_i$ the minimal sublattice of $L(P)$ containing $\{\alpha_i \mid i \in V(C_i)\}$ for each $i$. 
  It is sufficient to prove that $\calF_{L_i}$ forms a Khovanskii basis of $\calR(L_i)$ for each $i$ (see Lemma~\ref{lemma:main lemma}). 
  To this end, we may prove that each $L_i$ is a divisor lattice of $2 \cdot 3^m$ for some $m$  
  since we already know that the set of polynomials arising from these lattices forms a Khovanskii basis by Example~\ref{example:divisor lattice of 2*3^m}. 

  We pay attention to a corner of generalized snake poset as depicted in Figure~\ref{fig:A corner of generalized snake poset}.
  We observe that if we delete $\alpha$, then the co-comparability graph is separated into two disconnected graphs, 
  consisting of the vertices corresponding to elements larger than $\beta$ and smaller than $\beta$. 
  We denote them by $G_1$ and $G_2$, respectively. In $G_{L(P)}$, $G_1$ and $G_2$ are connected through a unique vertex $\alpha$. 
  Thus, there is no cycle that includes vertices from both $G_1$ and $G_2$. 
  Therefore, each $L_i$ is a divisor lattice of $2 \cdot 3^m$ for some $m$, as required. 

  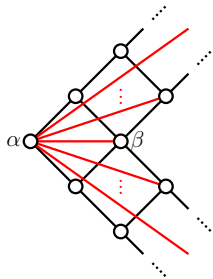
\begin{figure}[ht]
    \centering
    {\scalebox{0.6}{
      \begin{tikzpicture}[line width=0.05cm]

        \coordinate (N1) at (-1,1);  
        \coordinate (N2) at (0,2);   
        \coordinate (N3) at (-2,2);  
        \coordinate (N4) at (-1,3);  
        \coordinate (N5) at (-3,3);  
        \coordinate (N6) at (-2,4);  
        \coordinate (N7) at (0,4);   
        \coordinate (N8) at (-1,5);  
      
        \draw (N1)--(N2); \draw (N1)--(N3); \draw (N2)--(N4); \draw (N3)--(N4); 
        \draw (N3)--(N5); \draw (N5)--(N6); \draw (N4)--(N6); 
        \draw (N4)--(N7); \draw (N7)--(N8); \draw (N6)--(N8); 
    
        \draw (N8)--(-0.5, 5.5); \draw [dotted] (-0.3, 5.7)--(0,6);
        \draw (N7)--(0.5, 4.5); \draw [dotted] (0.7, 4.7)--(1,5);
        \draw (N1)--(-0.5, 0.5); \draw [dotted] (-0.3, 0.3)--(0,0);
        \draw (N2)--(0.5, 1.5); \draw [dotted] (0.7, 1.3)--(1,1);
    
        \draw [red] (N5)--(N2); \draw [red] (N5)--(N4); \draw [red] (N5)--(N7);
        \draw [red] (N5)--(0.5, 0.5); \draw [red] (N5)--(0.5, 5.5);
        \node [red] at (-1, 2.1) {\LARGE $\vdots$}; \node [red] at (-1, 4.1) {\LARGE $\vdots$};

        \draw [line width=0.05cm, fill=white] (N1) circle [radius=0.15];
        \draw [line width=0.05cm, fill=white] (N2) circle [radius=0.15];
        \draw [line width=0.05cm, fill=white] (N3) circle [radius=0.15];
        \draw [line width=0.05cm, fill=white] (N4) circle [radius=0.15] node [right=2pt] {\LARGE $\beta$};
        \draw [line width=0.05cm, fill=white] (N5) circle [radius=0.15] node [left=2pt] {\LARGE $\alpha$};
        \draw [line width=0.05cm, fill=white] (N6) circle [radius=0.15];
        \draw [line width=0.05cm, fill=white] (N7) circle [radius=0.15];
        \draw [line width=0.05cm, fill=white] (N8) circle [radius=0.15];
    
    \end{tikzpicture}
    }}
    \caption{A corner of generalized snake poset}
    \label{fig:A corner of generalized snake poset}
  \end{figure}

\bibliography{references}
\bibliographystyle{amsplain}
\end{document}